\documentclass[reqno,a4paper,draft]{amsart}

\usepackage{enumitem}
\setenumerate{label=\textnormal{(\arabic*)}}

\usepackage{amsmath,amssymb,dsfont,verbatim,bm,array,mathtools}
\usepackage[latin1]{inputenc}
\usepackage{booktabs}
\usepackage[raggedright]{titlesec}
\usepackage{mathtools}

\titleformat{\chapter}[display]
{\normalfont\huge\bfseries}{\chaptertitlename\\thechapter}{20pt}{\Huge}
\titleformat{\section}
{\normalfont\Large\bfseries\center}{\thesection}{1em}{}
\titleformat{\subsection}
{\normalfont\large\bfseries}{\thesubsection}{1em}{}
\titleformat{\subsubsection}[runin]
{\normalfont\normalsize\bfseries}{\thesubsubsection}{1em}{}
\titleformat{\paragraph}[runin]
{\normalfont\normalsize\bfseries}{\theparagraph}{1em}{}
\titleformat{\subparagraph}[runin]
{\normalfont\normalsize\bfseries}{\thesubparagraph}{1em}{}
\titlespacing*{\chapter} {0pt}{50pt}{40pt}
\titlespacing*{\section} {0pt}{3.5ex plus 1ex minus .2ex}{2.3ex plus .2ex}
\titlespacing*{\subsection} {0pt}{3.25ex plus 1ex minus .2ex}{1.5ex plus .2ex}
\titlespacing*{\subsubsection}{0pt}{3.25ex plus 1ex minus .2ex}{1.5ex plus .2ex}
\titlespacing*{\paragraph} {0pt}{3.25ex plus 1ex minus .2ex}{1em}
\titlespacing*{\subparagraph} {\parindent}{3.25ex plus 1ex minus .2ex}{1em}

\input xypic
\xyoption{all}

\subjclass[2010]{Primary 16W20, 14R15; Secondary 16S32}

\newtheorem{theorem}{Theorem}[section]
\newtheorem{lemma}[theorem]{Lemma}
\newtheorem{proposition}[theorem]{Proposition}
\newtheorem{corollary}[theorem]{Corollary}

\theoremstyle{definition}
\newtheorem{definition}[theorem]{Definition}
\newtheorem{notations}[theorem]{Notations}

\newtheorem{example}[theorem]{Example}

\theoremstyle{remark}
\newtheorem{remark}[theorem]{Remark}

\DeclareMathOperator{\Dir}{Dir}

\DeclareMathOperator{\Supp}{Supp}

\DeclareMathOperator{\en}{en}

\DeclareMathOperator{\st}{st}

\DeclareMathOperator{\Succ}{Succ}

\DeclareMathOperator{\Char}{Char}

\DeclareMathOperator{\Jac}{Jac}
\newcommand{\ov}{\overline}

\newcommand\sg[1]{{\left<{#1}\right>}}

\begin{document}
\title{The starred Dixmier conjecture for $A_1$}

\author{Vered Moskowicz}
\address{Department of Mathematics, Bar-Ilan University, Ramat-Gan 52900, Israel.}
\email{vered.moskowicz@gmail.com}
\thanks{Vered Moskowicz was supported partially by an Israel-US BSF grant \#2010/149}
\author{Christian Valqui}
\address{Pontificia Universidad Cat\'olica del Per\'u PUCP, Av. Universitaria 1801, 
San Miguel, Lima 32, Per\'u.}

\address{Instituto de Matem\'atica y Ciencias Afines (IMCA) Calle Los Bi\'ologos 245. Urb San C\'esar. 
La Molina, Lima 12, Per\'u.}
\email{cvalqui@pucp.edu.pe}
\thanks{Christian Valqui was supported by PUCP-DGI-2012-0011 and PUCP-DGI-2013-3036}
\begin{abstract}
Let $A_1(K)=K \langle x,y | yx-xy= 1 \rangle$ be the first Weyl algebra over a characteristic zero field 
$K$ and let
$\alpha$ be the exchange involution on $A_1(K)$ given by $\alpha(x)= y$ and $\alpha(y)= x$. 
The Dixmier conjecture of Dixmier (1968)
asks: Is every algebra endomorphism of the Weyl algebra $A_1(K)$ an automorphism?
The aim of this paper is to prove that each $\alpha$-endomorphism of $A_1(K)$ is an automorphism.
Here an $\alpha$-endomorphism of $A_1(K)$ is an endomorphism which preserves the involution $\alpha$. 
We also prove an analogue result for the Jacobian conjecture in dimension 2, called $\alpha-JC_2$.
\end{abstract}

\maketitle

\section{Introduction}
By definition, the $n$'th Weyl algebra $A_n(K)= A_n$ is the unital associative $K$-algebra generated 
by $2n$ elements
$x_1, \ldots, x_n, y_1, \ldots, y_n$ subject to the following defining relations:
$[y_i,x_j]= \delta_{ij}$, $[x_i,x_j]= 0$ and $[y_i,y_j]= 0$, where $\delta_{ij}$ is the Kronecker delta.

Here we will only deal with the first Weyl algebra $A_1(K)= K\sg{x,y | yx-xy= 1}$, where $\Char(K)= 0$.
(Only in Proposition \ref{main result} the field $K$ is not necessarily of zero characteristic).

In~\cite{adja} Adjamagbo and van den Essen remarked that $A_1$ was first studied by Dirac 
in~\cite{dirac}. Hence, they
suggest to call it ``Dirac quantum algebra" instead of (first) Weyl algebra. Similarly, 
they suggest to call $A_n$  ``$n$'th
Dirac quantum algebra" instead of $n$'th Weyl algebra. We truely do not know which name 
is better. For convenience,  in order to
maintain the same terminology used by most of the authors, we shall
continue to call $A_1$ the first Weyl algebra (and $A_n$ the $n$'th Weyl algebra).

In \cite{dixmier}, Dixmier asked six questions about the first Weyl algebra $A_1(K)$, where $K$ 
is a zero characteristic field;
the first question is the following: Is every algebra endomorphism of $A_1(K)$ an automorphism?

Usually, Dixmier's first question is brought as a conjecture; namely, the Dixmier conjecture says 
that every algebra endomorphism of
$A_1(K)$ is an automorphism.

In order to define a $K$-algebra homomorphism $f: A_1 \longrightarrow A_1$, it is enough to fix 
$f(x)$ and $f(y)$ such that
$[f(y),f(x)]= f(y)f(x)- f(x)f(y)= 1$, and extend it as an algebra homomorphism, and similarly 
for an antihomomorphism.

Recall that an antihomomorphism satisfies $f(ab)= f(b)f(a)$ and notice that the mapping
$\alpha: A_1 \longrightarrow A_1$ defined by $\alpha(x)= y$ and $\alpha(y)= x$, is an involution on $A_1$.
Indeed, $\alpha$ is an antihomomorphism of order $2$, so it is an antiautomorphism of order
$2$. This mapping $\alpha$ is sometimes called the exchange involution.

Of course, there are other involutions on $A_1$. For example, given any automorphism $g$ of $A_1$, 
$g^{-1} \alpha g$ is
clearly an involution on $A_1$.

Generally, it is easy to see that each involution on $A_1$ is of the form $h \alpha$, where $h$ 
is an automorphism of
$A_1$ which satisfies the following condition $h \alpha h \alpha= 1$.
Indeed, let $\beta$ be any involution on $A_1$. Then $\beta \alpha$ is an automorphism of $A_1$, 
call it $h$. From
$\beta \alpha= h$ follows $\beta= h \alpha$. Of course, since $\beta^2= 1$, we get $h \alpha h \alpha= 1$.

\begin{definition}\label{defin alpha endo}
An $\alpha$-endomorphism $f$ of $A_1$ is an endomorphism of $A_1$ which preserves the involution $\alpha$.
Preserving the involution $\alpha$ means that for every $w \in A_1$, $f(\alpha(w))= \alpha(f(w))$.
So an $\alpha$-endomorphism of $A_1$, $f$, is an endomorphism of $A_1$ which commutes with $\alpha$
($f \circ \alpha= \alpha \circ f$).
\end{definition}

It is easy to see that
 $$
 f(\alpha(w))= \alpha(f(w)),\ \forall w\in A_1 \quad \Longleftrightarrow\quad 
 f(\alpha(x))= \alpha(f(x)),\quad f(\alpha(y))= \alpha(f(y)).
$$
Therefore, $f$ is an $\alpha$-endomorphism of $A_1$, if $f$ is an endomorphism of $A_1$, for which
$f(\alpha(x))= \alpha(f(x))$ and $f(\alpha(y))= \alpha(f(y))$.

Now, one may pose the ``$\alpha$-Dixmier conjecture" or the ``starred Dixmier conjecture": Every
$\alpha$-endomorphism
of $A_1(K)$ ($\Char(K)=0$) is an automorphism.

\begin{remark}
The exchange involution $\alpha$ may be denoted by ``$*$ on the right", instead of ``$\alpha$ on the left"
(namely,
$x^*= y$ and $y^*= x$ instead of $\alpha(x)= y$ and $\alpha(y)= x$), hence the name the ``starred Dixmier
conjecture".
\end{remark}

In the following example we describe a family of $\alpha$-endomorphisms of any given even degree $2n$ (this
family also includes
degree $1$ $\alpha$-endomorphisms),
which is actually a family of $\alpha$-automorphisms:

\begin{example}\label{family of alpha auto} Let $n\in \mathds{N}_0$ and $a,b,c_0,\dots,c_n\in K$, with 
$a^2 - b^2 = 1$, and define
$S_j:=c_{j}(x-y)^{2j}$.
The following $f$ defines an $\alpha$-automorphism:
$$
f(x):= ax+ by+ \sum_{j= 0}^n S_j\quad\text{and}\quad f(y):= ay+ bx+ \sum_{j= 0}^n S_j.
$$
In fact, clearly $\alpha(S_j)=S_j$ and so $f(\alpha(x))= \alpha(f(x))$ and $f(\alpha(y))= \alpha(f(y))$.
Moreover,
\begin{align*}
    [f(y),f(x)]&=[ ay+ bx,ax+ by]+\! \left[ ay+ bx\pm ax,\sum_{j= 0}^n S_j\right]+
    \left[\sum_{j= 0}^n S_j,ax+ by\pm bx\right] \\
    &=a^2-b^2+\left[ (a+ b)x,\sum_{j= 0}^n S_j\right]+\left[\sum_{j= 0}^n S_j,(a+b)x\right]=1,
\end{align*}
where the second equality follows from the fact that $[x-y,S_j]=0$.
A straightforward computation shows that for
$$
 f_2:x\mapsto \frac x{a-b},\ y\mapsto (a-b)y,\quad f_3:x\mapsto x+y,\ y\mapsto y,\quad 
 f_4: x\mapsto x-y,\  y\mapsto y
$$
and
$$
f_1:x\mapsto x,\ y\mapsto y-\frac{bx}{a-b}-\sum_{j=0}^n c_j\left(\frac x{a-b}\right)^{2j},
$$
we have $id=f_1\circ f_2 \circ f_3 \circ f \circ f_4$ and so $f^{-1}=f_4\circ f_1\circ f_2\circ f_3$,
explicitly
$$
f^{-1}(x)=ax-by+(b-a)\sum_{j= 0}^n \tilde{S}_j\quad\text{and}\quad f^{-1}(y)=ay-bx+(b-a)\sum_{j= 0}^n
\tilde{S}_j,
$$
where $\tilde S_j:=c_j\left(\displaystyle\frac{x-y}{a-b}\right)^{2j}$.
\end{example}
Surprisingly, the family of $\alpha$-automorphisms of Example \ref{family of alpha auto} describes all the
$\alpha$-endomorphisms
of $A_1(K)$ (see Corollary \ref{main corollary}). Therefore, the starred Dixmier conjecture is true.

\section{Proof of the starred Dixmier conjecture}
Consider the automorphism $\varphi$ of $A_1$ given by $\varphi(x):=\displaystyle\frac{x+y}2$ and 
$\varphi(y):= y-x$. Then
$\varphi^{-1}(x)=x-\displaystyle\frac y2$,
$\varphi^{-1}(y)=x+\displaystyle\frac y2$ and
the antihomomorphism $\beta:=\varphi^{-1}\circ \alpha \circ\varphi$ is given by $\beta(x)=x$ and 
$\beta(y)= -y$.
Notice that $\beta$ is an involution.

For the rest of the section we fix an $\alpha$-endomorphism  $f$. Then the endomorphism 
$\tilde f:=\varphi^{-1}\circ f$ satisfies
$$
\beta(\tilde f(x))=\varphi^{-1}\circ \alpha \circ\varphi(\varphi^{-1} \circ f(x))=\varphi^{-1} 
\alpha( f(x))=\varphi^{-1}( f(\alpha(x)))=\tilde f(y)
$$
and similarly $\beta(\tilde f(y))=\tilde f(x)$. (Note that $\tilde f$ is NOT a $\beta$-endomorphism).

\noindent We set $P:=\tilde f(x)$ and $Q:=\tilde f(y)$. Since $\tilde f$ is an endomorphism, we have $[Q,P]=1$.
We decompose $P$ into its symmetric and antisymmetric terms with respect to $\beta$:
$P= P_0+P_1$ with $P_0=(P+\beta(P))/2$ and
$P_1=(P-\beta(P))/2$. Note that $\beta(P_0)=P_0$ and $\beta(P_1)=-P_1$. Now $\beta(P)=Q$ implies $Q=P_0-P_1$,
and so
$$
1=[Q,P]=[P_0-P_1,P_0+P_1]=2[P_0,P_1],
$$
hence
$[P_0,P_1]=1/2$.

We recall some definitions and notations of~\cite{GGV}. Let $L:=K[X,Y]$ be the polynomial $K$-algebra in two
variables and let $\Psi\colon A_1\to L$
be the $K$-linear map defined by $\Psi\bigl(x^iy^j\bigr) := X^i Y^j$. Let
\begin{align*}
& \ov{\mathfrak{V}} := \{(\rho,\sigma)\in \mathds{Z}^2: \text{$\gcd(\rho,\sigma) = 1$ and $\rho+\sigma\ge 0$}\}
\intertext{and}
&\mathfrak{V} := \{(\rho,\sigma)\in \ov{\mathfrak{V}}: \rho+\sigma> 0 \}.
\end{align*}
Note that $\ov{\mathfrak{V}} = \mathfrak{V}\cup \{(1,-1),(-1,1)\}$.

\begin{definition} For all $(\rho,\sigma)\in \ov{\mathfrak{V}}$ and $(i,j)\in \mathds{Z}\times \mathds{Z}$ we
write
$$
v_{\rho,\sigma}(i,j):= \rho i + \sigma j.
$$
\end{definition}

\begin{notations}\label{not valuaciones para polinomios} Let $(\rho,\sigma)\in \ov{\mathfrak{V}}$. For 
$P = \sum a_{ij} X^i Y^j\in L\setminus\{0\}$, we define:

\begin{itemize}

\smallskip

\item[-] The {\em support} of $P$ as
$$
\Supp(P) := \left\{(i,j): a_{ij}\ne 0\right\}.
$$

\smallskip

\item[-] The {\em $(\rho,\sigma)$-degree} of $P$ as $v_{\rho,\sigma}(P):= 
\max\left\{ v_{\rho,\sigma}(i,j): a_{ij} \ne 0 \right\}$.

\smallskip

\item[-] The {\em $(\rho,\sigma)$-leading term} of $P$ as
$$
\ell_{\rho,\sigma}(P):= \displaystyle \sum_{\{\rho i + \sigma j = v_{\rho,\sigma}(P)\}} a_{ij} X^i Y^j.
$$

\smallskip

\item[-] $w(P):= \left(i_0,i_0-v_{1,-1}(P)\right)$ such that
$$
i_0 = \max \left\{i: \left(i,i-v_{1,-1}(P) \right) \in \Supp(\ell_{1,-1}(P))\right\},
$$

\smallskip

\item[-] $\ov{w}(P):= \left(i_0-v_{-1,1}(P),i_0\right)$ such that
$$
i_0 = \max \left\{i: \left(i-v_{-1,1}(P),i \right) \in \Supp(\ell_{-1,1}(P))\right\},
$$

\end{itemize}
\end{notations}

\begin{notations}\label{not valuaciones para alg de Weyl} Let $(\rho,\sigma)\in \ov{\mathfrak{V}}$. For 
$P \in A_1\setminus\{0\}$, we define:

\begin{itemize}

\smallskip

\item[-] The {\em support} of $P$ as $\Supp(P) := \Supp\bigl(\Psi(P)\bigr)$.

\smallskip

\item[-] The {\em $(\rho,\sigma)$-degree} of $P$ as $v_{\rho,\sigma}(P):= v_{\rho,\sigma}
    \bigl(\Psi(P)\bigr)$.

\smallskip

\item[-] The {\em $(\rho,\sigma)$-leading term} of $P$ as $\ell_{\rho,\sigma}(P):= \ell_{\rho,\sigma}
    \bigl(\Psi(P)\bigr)$.

\smallskip

\item[-] $w(P):= w\bigl(\Psi(P)\bigr)$ and $\ov{w}(P):= \ov{w}\bigl(\Psi(P)\bigr)$.

\end{itemize}
\end{notations}

\begin{definition}\label{Comienzo y Fin de un elemento de W^{(l)}} Let $(\rho,\sigma)\in \ov{\mathfrak{V}}$ 
and let $P\in W\setminus\{0\}$.

\begin{itemize}

\smallskip

\item[-] If $(\rho,\sigma)\ne (1,-1)$, then the {\em starting point of $P$ with respect to 
$(\rho,\sigma)$} is
$$
\st_{\rho,\sigma}(P) = w(\ell_{\rho,\sigma}(P)).
$$

\smallskip

\item[-] If $(\rho,\sigma)\ne (-1,1)$, then the {\em end point of $P$ with respect to $(\rho,\sigma)$} is
$$
\en_{\rho,\sigma}(P) = \ov{w}(\ell_{\rho,\sigma}(P)).
$$
\end{itemize}
\end{definition}
\begin{lemma}\label{standard form}
Let $P,Q\in A_1$ with $[Q,P]=1$. If $\ell_{1,0}(Q)=\lambda X^s Y$ for some $\lambda\in K^*$, then $s=0$.
\end{lemma}

\begin{proof}
Assume by contradiction that $s>0$ and take $T_0$ such that $[Q,T_0]=1$ and
$$
v_{1,0}(T_0)=\min\{ v_{1,0}(T) | \ [Q,T]=1, \ T\in A_1\}.
$$
Notice that $v_{1,0}(T_0)> 0$, since otherwise $T_0\in K[y]$ and then
$T_0=p(y)$ for some polynomial $p(y)$ and $[Q,T_0]_{1,0}= - \lambda p'(y)y\ne \ell_{1,0}(1)$, a contradiction.

If $[Q,T_0]_{1,0}\ne 0$ (See~\cite[Definition 2.2]{GGV}), then we obtain the contradiction
$$
0=v_{1,0}([Q,T_0])=v_{1,0}(Q)+v_{1,0}(T_0)-1>0.
$$

On the other hand, if $[Q,T_0]_{1,0}= 0$ then by~\cite[Theorem 2.11]{GGV} there exist $R\in L$,
$\lambda_Q,\lambda_P\in K^*$ and $n,m\in\mathds{N}$
such that
$$
\ell_{1,0}(Q)=\lambda_Q R^n\quad \text{and}\quad \ell_{1,0}(T_0)=\lambda_t R^m.
$$
Clearly $n=1$ and so $\ell_{1,0}(T_0)=\frac{\lambda_t }{\lambda_Q^m}(\ell_{1,0}(Q))^m$, hence
$T_1:=T_0-\frac{\lambda_t }{\lambda_Q^m} Q^m$ satisfies $[Q,T_1]=1$ and $v_{1,0}(T_1)<v_{1,0}(T_0)$, which
contradicts the minimality
of $v_{1,0}(T_0)$ and concludes the proof.
\end{proof}

In the following proposition $K$ is any field with $\Char(K)\neq 2$, not necessarily a zero characteristic
field.

\begin{proposition}\label{main result}
Let $P_0,P_1\in A_1(K)$. Assume $2r\ne 0$ in $K$, where $r=\deg_x(P_0)$.
If $\beta(P_0)=P_0$, $\beta(P_1)=-P_1$ and $[P_0,P_1]=\frac 12$, then
$$
P_1= \lambda y\quad\text{and}\quad P_0= -\frac{1}{2\lambda} x+\sum_{j=0}^n \alpha_j y^{2j},\quad \text{
for some $\lambda\ne 0$ and $\alpha_j\in K$.}
$$
\end{proposition}

\begin{proof}
Write $\ell_{1,0}(P_0)=X^rg(Y)$ and $\ell_{1,0}(P_1)=X^s h(Y)$. Then
$$
\ell_{1,0}(\beta(P_0))=\ell_{1,0}(g(-y)x^r)=X^r g(-Y),
$$
and so $g(-y)=g(y)$, i.e., $g$ is even, which means $g(y)\in K[y^2]$.
Similarly
$$
\ell_{1,0}(\beta(P_1))=\ell_{1,0}(h(-y)x^s)=X^s h(-Y),
$$
and so $h(-y)=-h(y)$, i.e., $h$ is odd, which means $h(y)\in yK[y^2]$.

Assume $s=0$, then $\en_{1,0}(P_1)=(0,2k+1)$ for some $k\in\mathds{N}_0$. If 
$\en_{1,0}(P_1)\sim \en_{1,0}(P_0)$ then $P_0\in K[y]$ which
leads to the contradiction $[P_0,P_1]=0$. Hence $\en_{1,0}(P_1)\nsim \en_{1,0}(P_0)$ and so, 
by~\cite[Corollary 2.7]{GGV} we have
$[P_0,P_1]_{1,0}\ne 0$ and by~\cite[Proposition 2.4]{GGV}
$$
\en_{1,0}(P_0)+\en_{1,0}(P_1)-(1,1)=\en_{1,0}([P_0,P_1])=(0,0).
$$
It follows that $k=0$, $\en_{1,0}(P_0)=(1,0)$ and $P_1=\lambda y$ for some $\lambda\in K^*$. But then
$P_0= -\frac{1}{2\lambda} x+t(y)$ for some $t(y)\in K[y]$ and from $\beta(P_0)=P_0$ we deduce 
$t(y)=\sum_{j=0}^n \alpha_j y^{2j}$
for some $\alpha_j\in K$.

In order to finish the proof it suffices to discard the case $s>0$. So assume $s>0$ and note that the
assumptions require $r>0$.
We will consider the two highest order terms in $\deg_x$. By~\cite[Proposition 1.6]{GGV} for all 
$t(y)\in K[y]$ and $i\in \mathds{N}$ we have
\begin{equation}\label{conmutator}
    [t(y),x^i]=i x^{i-1}t'(y)+\binom{i}{2} x^{i-2}t''(y)+\dots,
\end{equation}
where from now on ``$\dots$'' denotes terms of lower $x$-degree.
Write now
$$
P_0=x^r g(y)+x^{r-1}g_1(y)+\dots\quad\text{and}\quad P_1=x^s h(y)+x^{s-1}h_1(y)+\dots.
$$
Then
$$
\beta(P_0)=g(y)x^r +g_1(-y)x^{r-1}+\ldots =x^r g(y)+r x^{r-1}g'(y) +x^{r-1}g_1(-y)+\dots
$$
From $\beta(P_0)=P_0$ we deduce $g_1(y)=rg'(y)+g_1(-y)$. Decomposing $g_1$ into its even and odd parts we
obtain
$g_1(y)=g_{10}(y)+g_{11}(y)$ with $g_{10}\in K[y^2]$ and $g_{11}\in yK[y^2]$ and so
\begin{equation}\label{decomposition of g1}
    g_1(y)=g_{10}(y) +\frac r2 g'(y).
\end{equation}
Similarly we obtain
\begin{equation}\label{decomposition of h1}
    h_1(y)=\frac s2 h'(y)+h_{11}(y)\quad\text{for some $h_{11}\in yK[y^2]$.}
\end{equation}
From~\eqref{conmutator} we deduce
\begin{align*}
[&P_0,P_1]=x^{r+s-1}(sg'h-rh'g)\\
&+x^{r+s-2}\left(\binom{s}{2}g''h-\binom{r}{2}h''g+(s-1)g'h_1-rh_1'g+sg_1'h-(r-1)h'g_1\right)+\dots
\end{align*}
We insert in the coefficient corresponding to $x^{r+s-2}$ the values of $g_1$ and $h_1$ according
to~\eqref{decomposition of g1}
and~\eqref{decomposition of h1}:
\begin{align*}
    \binom{s}{2}&g''h-\binom{r}{2}h''g+(s-1)g'h_1-rh_1'g+sg_1'h-(r-1)h'g_1=\\
    =&(r+s-1)\left( sg''h+sg'h'-rh'g'-rh''g\right)+2r gh''+\\
    &+ (s-1)g'h_{11}-r h_{11}'g+s g_{10}'h-(r-1)h'g_{10}.
\end{align*}
Now $sg'h-rh'g=0$ implies $sg''h+sg'h'-rh'g'-rh''g=0$ and moreover
$$
(s-1)g'h_{11}-r h_{11}'g+s g_{10}'h-(r-1)h'g_{10}
$$
 is even, hence
the odd part of the coefficient is $2rgh''=0$. By assumption $2r\ne 0$ and $g\ne 0$, which leads to $h''=0$. 
But then
$h=\lambda y$ for some $\lambda\in K^*$. By Lemma~\ref{standard form} this implies $s=0$, a contradiction 
which concludes the proof.
\end{proof}

\begin{corollary}\label{corollary of main result}
Assume $\Char(K)=0$. For any $\alpha$-endomorphism $f$ of $A_1(K)$ there exist  $n\in\mathds{N}_0$,
$\lambda\in K^*$ and $\alpha_0,\dots,\alpha_n\in K$ such that
$$
f(x)=\lambda(y-x)-\frac{y+x}{4 \lambda}+\sum_{j=0}^n \alpha_j (y-x)^{2j}\ ,\
f(y)=- \lambda(y-x)-\frac{y+x}{4 \lambda}+\sum_{j=0}^n \alpha_j (y-x)^{2j}.
$$
\end{corollary}
\begin{proof} From the discussion above Lemma~\ref{standard form} we know that
$$
f(x)=\varphi \tilde f(x)=\varphi(P_0+P_1)\quad\text{and}\quad f(y)=\varphi \tilde f(y)=\varphi(P_0-P_1),
$$
for some $P_0,P_1$ satisfying $\beta(P_0)=P_0$, $\beta(P_1)=-P_1$ and $[P_0,P_1]=\frac 12$. Moreover
$r:=\deg_x(P_0)=0$ is impossible since it would lead to $P_0\in K[y^2]$ which implies $[P_0,P_1]\notin K^*$
(compute $\ell_{1,0}[P_0,P_1]$).
Hence $2r\ne 0$ and by Proposition~\ref{main result}
we know that there exist $n\in\mathds{N}_0$,
$\lambda\in K^*$ and $\alpha_0,\dots,\alpha_n\in K$ such that
$P_1=\lambda y$ and $P_0= -\frac{1}{2\lambda} x+\sum_{j=0}^n \alpha_j y^{2j}$. A straightforward computation
concludes the proof.
\end{proof}

\begin{corollary}\label{main corollary}
Assume $\Char(K)=0$. Any $\alpha$-endomorphism $f$ of $A_1(K)$ is an $\alpha$-automorphism of the same form as
in Example~\ref{family of alpha auto}.
\end{corollary}
\begin{proof}
Let $f$ be an $\alpha$-endomorphism. Apply Corollary \ref{corollary of main result} to $f$ to get
$n\in\mathds{N}_0$,
$\lambda\in K^*$ and $\alpha_0,\dots,\alpha_n\in K$ such that
$$
f(x)=\lambda(y-x)-\frac{y+x}{4 \lambda}+\sum_{j=0}^n \alpha_j (y-x)^{2j}\ \text{,}\quad
f(y)=- \lambda(y-x)-\frac{y+x}{4 \lambda}+\sum_{j=0}^n \alpha_j (y-x)^{2j}.
$$
Hence,
$$
f(x)= ax+by +\sum_{j=0}^n \alpha_j (y-x)^{2j}\quad \text{and}\quad
f(y)= ay+bx +\sum_{j=0}^n \alpha_j (y-x)^{2j},
$$
where $a= (-4\lambda^2-1)/ 4\lambda$ and $b= (4\lambda^2-1)/ 4\lambda$ satisfy $a^2 -b^2=1$,
and so the endomorphism is of the desired form.
\end{proof}
\begin{theorem}[The starred Dixmier conjecture is true]
Assume $\Char(K)=0$. Any $\alpha$-endomorphism $f$ of $A_1(K)$ is an $\alpha$-automorphism.
\end{theorem}\begin{proof} Clear from Corollary~\ref{main corollary} and 
Example~\ref{family of alpha auto}.\end{proof}

\section{Related topics}
We suggest to consider the following topics. In those topics we usually take the exchange involution $\alpha$,
although one may
take other involutions as well.

\subsection{Prime characteristic case}

When $K$ is of prime characteristic, $A_1(K)$ is not simple. Bavula asked the following question: Is every
algebra endomorphism
of the first Weyl algebra $A_1(K)$, $\Char(K)= p > 0$, a monomorphism? Makar-Limanov \cite{makar BC} gave a
positive answer to
that question. However, the following endomorphism $f$
($f$ is necessarily a monomorphism) is not onto, since $x$ is not in the image of $f$:

$f: A_1 \longrightarrow A_1$ such that $f(x)= x+x^p$ and $f(y)= y$.

Note that Proposition~\ref{main result} is valid for $\Char(K)\ne 2$. This suggest that
the following is true (the starred Dixmier conjecture in the
positive characteristic case, compare with~\cite{adja}):

``Let $\Char(K)=p>2$. Any $\alpha$-endomorphism of $A_1(K)$ is an automorphism if and only if its
restriction to the center of $A_1$ induces a field extension of degree not
a multiple of $p$ and the jacobian of this restriction is a nonzero element of
$K$."

\subsection{Higher Weyl algebras}

Assume $\Char(K)= 0$ and let $A_n(K)$ be the $n$'th Weyl algebra and $\alpha:A_n\to A_n$ an involution, for
example the anti-homomorphism given
by $x_i\mapsto y_i$, $y_i\mapsto x_i$. If one generalizes the geometric methods of~\cite{GGV} to higher
dimensions, one could try
to prove the $n$'th starred Dixmier conjecture ($\alpha-D_n$):
``Any $\alpha$-endomorphism of $A_n$ is an automorphism."

\subsection{Same questions for other algebras}
One may wish to ask similar questions for other algebras, see, for example, \cite{bavula analogue}.
In algebras where an involution can be defined, one may wish to see if the presence of an involution 
may be of any help in solving such questions.

\subsection{Connection to the Jacobian conjecture}
The connection between Dixmier's problem $1$ and the Jacobian conjecture is as follows:
\begin{itemize}
\item [(1)] The $n$'th Dixmier conjecture, $D_n$
(Every endomorphism of $A_n(K)$ is an automorphism),
implies $JC_n$, the $n$-dimensional Jacobian conjecture, 
see~\cite[Theorem 4.2.8]{jacobian van den essen}.

\item [(2)] $JC_{2n}$ $\Rightarrow$ $D_n$.
\end{itemize}
Interestingly, the Jacobian conjecture-$2n$ implies the $n$'th Dixmier conjecture.
This was proved independently by Tsuchimoto \cite{tsuchimoto} and by Belov and Kontsevich \cite{belov}.
A shorter proof can be found in \cite{bavula jacobian}.
For a detailed background on the Jacobian conjecture, see, for example,~\cite{bass connell wright} 
or~\cite{jacobian van den essen}.
One may pose the $\alpha$-Jacobian conjecture ($\alpha$-JC$_{2n}$) in dimension $2n$: ``Let
$\alpha:X_{2n-1}\leftrightarrow X_{2n}$
be the exchange involution on $K[X_1,\dots,X_{2n}]$. If $f: K[X_1,\dots,X_{2n}] \to K[X_1,\dots,X_{2n}]$ 
is an $\alpha$-morphism which satisfies
$\Jac(f(X_1),\dots,f(X_{2n}))=1$, then $f$ is invertible."

One may ask the following question:  $\alpha-JC_{2n}$ $\Rightarrow$ $\alpha-D_{n}$ ?

In the previous section we proved $\alpha-D_1$ and in the last section we will prove $\alpha-JC_{2}$; 
however we don't know if
one can deduce $\alpha-D_1$ directly from $\alpha-JC_{2}$.

\section{Analogue results for the Jacobian conjecture}

There is a closed connection between the shape of possible counterexamples to the Jacobian conjecture in
dimension 2 and
the shape of possible counterexamples to the Dixmier conjecture (in dimension 1). Let 
$\alpha:K[X,Y]\to K[X,Y]$ be the
exchange involution $X\leftrightarrow Y$. We say that $f:K[X,Y]\to K[X,Y]$ is an $\alpha$-morphism, if 
$f \circ \alpha= \alpha \circ f$.
Note that the symmetry determined by $\alpha$
is not any of the symmetries analyzed in~\cite{debondt} or~\cite{vdEW}.

One has the following result:
\begin{proposition}($\alpha-JC_{2}$ is true.)\label{JC2}
If $K$ is a field of characteristic zero and $f:K[X,Y]\to K[X,Y]$ is an $\alpha$-morphism that satisfies
$\Jac(f(X),f(Y))=1$,
then $f$ is an automorphism.
\end{proposition}
\begin{proof}
Let $\varphi:K[X,Y]\to K[X,Y]$ be given by $\varphi(X):=\displaystyle\frac{X+Y}2$ and $\varphi(Y):= Y-X$. Then
$\varphi^{-1}(X)=X-\displaystyle\frac Y2$,
$\varphi^{-1}(Y)=X+\displaystyle\frac Y2$ and
the homomorphism $\beta:=\varphi^{-1}\circ \alpha \circ\varphi$ is given by $\beta(X)=X$ and $\beta(Y)= -Y$.
Notice that $\beta$ is an involution which commutes with $\tilde f:=\varphi^{-1}\circ f\circ \varphi$, 
since $f$ is an $\alpha$-morphism.
This means that
$P:=\tilde f(X)$ and $Q:=\tilde f(Y)$ satisfy
$$
\beta(P)=\beta(\tilde f(X))=\tilde f(\beta (X))=P,\quad \beta(Q)=\beta(\tilde f(Y))=\tilde f(\beta (Y))=-Q
$$
and $\Jac(P,Q)=1$.
We will prove that for $P,Q\in K[X,Y]$ with $\beta(P)=P$, $\beta(Q)=-Q$ and $\Jac(P,Q)=1$ there exist
$\lambda\in K^{\times}$ and $g\in K[Y^2]$
such that $Q=\lambda Y$ and $P=\frac{X}{\lambda}+g(Y)$.

This implies that $\tilde f$ is invertible, with $\tilde f^{-1}(X)=
\lambda X-\lambda g\left(\frac Y{\lambda}\right)$ and
$\tilde f^{-1}(Y)=\frac Y{\lambda}$; hence $f$ is also invertible.

We adopt the notations and results of~\cite{GGV1}, in particular $[P,Q]:=\Jac(P,Q)$. Let $P$ and $Q$ be as
above.
It is clear that $P\in K[X,Y^2]$ and $Q\in Y K[X,Y^2]$. Then
$w_{0,1}(P)\ge 0$ and $w_{0,1}(Q)\ge 1$ and so, since by~\cite[Proposition 1.11]{GGV1} we have
$$
0=w_{0,1}([P,Q])\ge w_{0,1}(P) + w_{0,1}(Q) - (0+1),
$$
it follows that $w_{0,1}(P)= 0$, $w_{0,1}(Q)= 1$ and $0\ne
[\ell\ell_{0,1}(P),\ell\ell_{0,1}(Q)]=\ell\ell_{0,1}([P,Q])$,
by the same proposition. Write
$\ell\ell_{0,1}(P)=h_p(X)$ and $\ell\ell_{0,1}(Q)=Yh_q(X)$, then
$$
1=[\ell\ell_{0,1}(P),\ell\ell_{0,1}(Q)] = h_p'(X)h_q(X),
$$
 hence
$$
\ell\ell_{0,1}(Q)=\lambda Y\quad\text{and}\quad\ell\ell_{0,1}(P)=\frac 1{\lambda}X,
$$
for some $\lambda\in K^{\times}$.
We prove now that
\begin{equation}\label{principal term}
\ell_{1,0}(Q)=\lambda Y\quad\text{and}\quad\ell_{1,0}(P)=\frac 1{\lambda}X.
\end{equation}
Let $(\rho,\sigma):=\min\{\underline{\Succ}_{P}(0,1),\underline{\Succ}_{Q}(0,1)\}$. We first prove that

\begin{equation}\label{first condition}
(\rho,\sigma)=(-1,1),\quad \ell_{1,-1}(P)=\frac 1{\lambda} X\quad\text{and}\quad \ell_{1,-1}(Q)=\lambda Y.
\end{equation}
Note that by~\cite[Proposition 3.13]{GGV1} we have $\underline{\st}_{\rho,\sigma}(P)=(1,0)$ and
$\underline{\st}_{\rho,\sigma}(Q)=(0,1)$. Moreover, using the definition of the order on directions
of~\cite[Section 3]{GGV1},
the inequalities $(-1,1)\ge(\rho,\sigma)> (0,1)$, imply $\sigma\ge -\rho>0$. Hence
$$
v_{\rho,\sigma}(\underline{\en}_{\rho,\sigma}(P))=v_{\rho,\sigma}(\underline{\st}_{\rho,\sigma}(P))=v_{\rho,\sigma}(1,0)=\rho<0
$$
and
$$
v_{\rho,\sigma}(\underline{\en}_{\rho,\sigma}(Q))=v_{\rho,\sigma}(\underline{\st}_{\rho,\sigma}(Q))=v_{\rho,\sigma}(0,1)=\sigma>0,
$$
and so
$\underline{\en}_{\rho,\sigma}(P)\nsim\underline{\en}_{\rho,\sigma}(Q)$, since two non-zero points $A,B$ 
in the first quadrant are
aligned if and only if  $A=\gamma B$ for some $\gamma>0$.
But then, by~\cite[Proposition 2.3(2)]{GGV1} we have 
$[\ell\ell_{\rho,\sigma}(P),\ell\ell_{\rho,\sigma}(Q)]\ne 0$ and so,
by~\cite[Proposition 2.5(2)]{GGV1},
$$
\underline{\en}_{\rho,\sigma}(P)+\underline{\en}_{\rho,\sigma}(Q)-(1,1)=\underline{\en}_{\rho,\sigma}([P,Q])=(0,0).
$$
But the only non-aligned points in $\mathds{N}_0\times \mathds{N}_0$ which sum $(1,1)$, are $(1,0)$ and 
$(0,1)$, hence
the only possibility is
$$
\underline{\en}_{\rho,\sigma}(P)=\underline{\st}_{\rho,\sigma}(P)=(1,0)
$$
and
$$
\underline{\en}_{\rho,\sigma}(Q)=\underline{\st}_{\rho,\sigma}(Q)=(0,1).
$$
Since $(\rho,\sigma)\in \underline{\overline{\Dir}}(P)\cup \underline{\overline{\Dir}}(Q)$ it follows that
$(\rho,\sigma)=(-1,1)$, $\ell_{1,-1}(P)=\frac 1{\lambda} X$ and $\ell_{1,-1}(Q)=\lambda Y$, which
is~\eqref{first condition}.

We claim that
\begin{equation}\label{second condition}
(\rho_1,\sigma_1):=\min\{\Succ_{P}(1,-1),\Succ_{Q}(1,-1)\}\ge(1,0).
\end{equation}
In fact the geometric argument is the same as before:
Assume $(\rho_1,\sigma_1)<(1,0)$.
Note that by~\cite[Proposition 3.12]{GGV1} we have $\st_{\rho_1,\sigma_1}(P)=(1,0)$ and
$\st_{\rho_1,\sigma_1}(Q)=(0,1)$. Moreover, using the definition of the order on directions 
of~\cite[Section 3]{GGV1},
the inequalities $(1,-1)\le(\rho_1,\sigma_1)< (1,0)$, imply $\rho_1\ge -\sigma_1>0$. Hence
$$
v_{\rho_1,\sigma_1}(\en_{\rho_1,\sigma_1}(P))=v_{\rho_1,\sigma_1}(\st_{\rho_1,\sigma_1}(P))=v_{\rho_1,\sigma_1}(1,0)=\rho_1>0
$$
and
$$
v_{\rho_1,\sigma_1}(\en_{\rho_1,\sigma_1}(Q))=v_{\rho_1,\sigma_1}(\st_{\rho_1,\sigma_1}(Q))=v_{\rho_1,\sigma_1}(0,1)=\sigma_1<0,
$$
and so
$\en_{\rho_1,\sigma_1}(P)\nsim\en_{\rho_1,\sigma_1}(Q)$, since two non-zero points $A,B$ in the 
first quadrant are
aligned if and only if  $A=\gamma B$ for some $\gamma>0$.
But then, by~\cite[Proposition 2.3(1)]{GGV1} we have 
$[\ell_{\rho_1,\sigma_1}(P),\ell_{\rho_1,\sigma_1}(Q)]\ne 0$ and so,
by~\cite[Proposition 2.4(2)]{GGV1},
$$
\en_{\rho_1,\sigma_1}(P)+\en_{\rho_1,\sigma_1}(Q)-(1,1)=\en_{\rho_1,\sigma_1}([P,Q])=(0,0).
$$
But the only non-aligned points in $\mathds{N}_0\times \mathds{N}_0$ which sum $(1,1)$, 
are $(1,0)$ and $(0,1)$, hence
the only possibility is
$$
\en_{\rho_1,\sigma_1}(P)=\st_{\rho_1,\sigma_1}(P)=(1,0)
$$
and
$$
\en_{\rho_1,\sigma_1}(Q)=\st_{\rho_1,\sigma_1}(Q)=(0,1).
$$
But this contradicts the fact that $(\rho_1,\sigma_1)\in \Dir(P)\cup \Dir(Q)$ and 
proves~\eqref{second condition}.

Now from~\eqref{second condition} and~\cite[Proposition 3.12]{GGV1}  it follows that $\st_{1,0}(Q)=(0,1)$ and
$\st_{1,0}(P)=(1,0)$, hence
$$
v_{1,0}(Q)=v_{1,0}(\st_{1,0}(Q))=0\quad\text{and}\quad v_{1,0}(P)=v_{1,0}(\st_{1,0}(P))=1,
$$
which implies
$$
Q=g_Q(Y)\quad\text{and}\quad P=g(Y)+X g_P(Y),\quad\text{for some $g_Q(Y) , g(Y),g_P(Y)\in K[Y],$}
$$
with $g(Y)\in K[Y^2]$. But then
$$
1=[P,Q]=g_P(Y) g_Q'(Y),
$$
hence $g_Q(Y)=\lambda Y$ and $g_P(Y)=\frac 1{\lambda}$, which proves~\eqref{principal term} 
and concludes the proof.
\end{proof}
\begin{remark}
From Proposition~\ref{JC2}, it can easily be seen that each $\alpha$-morphism $f$ which satisfies
$\Jac(f(X),f(Y))=1$ is of the following form:
$$
f(X):= aX+ bY+ \sum_{j= 0}^n T_j\quad\text{and}\quad f(Y):= aY+ bX+ \sum_{j= 0}^n T_j
$$
where $n\in \mathds{N}_0$, $a,b\in K$, with $a^2 - b^2 = 1$, and $T_j:=c_{j}(X-Y)^{2j}$ for some
$c_0,\dots,c_n\in K$.

One can structure the proof of $\alpha-JC_2$ similar to the proof of $\alpha-D_1$, however the key 
difference is the following:
In the proof of $\alpha-D_1$ we obtain a contradiction using the second highest order term
via~\eqref{conmutator},
which is not present in $\alpha-JC_2$, if we
identify the commutator with the Jacobian bracket. On the other hand, in the proof of $\alpha-JC_2$ 
we make use of $\ell\ell_{\rho,\sigma}$,
which is not well-behaved in $A_1(K)$.
\end{remark}
\bibliographystyle{plain}

\end{document}